\documentclass{amsart}
\usepackage{amsmath,amssymb,amsthm,mathtools,mathrsfs}
\usepackage{multirow,tikz}
\DeclarePairedDelimiter\norm{\lvert}{\rvert}
\DeclarePairedDelimiter\inner{\langle}{\rangle}
\usepackage{bbm}

\usepackage[colorlinks=true,linkcolor=blue,citecolor=blue]{hyperref}
\usepackage{enumitem}
\usepackage{csquotes}

\numberwithin{equation}{section}
\newcounter{intro}

		\newtheorem{introthm}[intro]{Theorem}

		\newtheorem{thm}[equation]{Theorem}
		\newtheorem{lem}[equation]{Lemma}
		\newtheorem{cor}[equation]{Corollary}

\theoremstyle{remark}
		
\theoremstyle{definition}
		
		\newtheorem{exam}[equation]{Example}

\def\xstrut{\rule{-4pt}{-4ex}}
\title[Supercharacter Theory via the Group Determinant]{Supercharacter Theory\\ via the Group Determinant}
\author{Shawn T. Burkett}
\address{Department of Mathematical Sciences, Kent State University, Kent,
Ohio 44240, U.S.A.} \email{sburket1@kent.edu}
\date{\today}
\begin{document}
\maketitle
\begin{abstract}
Ferdinand Georg Frobenius is generally considered the creator of character theory of finite groups. This achievement came from the study of the group determinant, which is the determinant of a matrix coming from the regular representation. In this paper, we generalize several of Frobenius' results about the group determinant and use them find a new formulation of supercharacter theory in terms of factorizations of the group determinant.
\end{abstract}
\section{Introduction}
Character theory of finite groups is a subject with a rich and fascinating origin story. Although characters of abelian groups were used at the time of Gauss, the notion of group characters for nonabelian groups did not exist until 1896, when Ferdinand Georg Frobenius discovered a suitable notion. To describe Frobenius' method, we must begin by defining the {\it group determinant}. Let $G$ be a finite group and let $x_g$ be a set of independent indeterminates indexed by the elements of $G$. The {\it group matrix} is the matrix whose rows and columns are indexed by the elements of $G$, and whose $(g,h)$-entry is $x_{gh^{-1}}$. The group determinant $\Theta(G)$ is the determinant of the group matrix, and is a homogeneous polynomial monic in $x_1$ (if there is no ambiguity about the group $G$, we will simply write $\Theta$). Analysis of $\Theta$ allowed Frobenius to develop a theory of group characters for nonabelian groups that generalized the existing notion for abelian groups.

Frobenius' work on the group determinant began after a private correspondence with Richard Dedekind, whose work Frobenius saught to generalize. Dedekind had been studying the group determinant of abelian groups and proved that $\Theta$ factors completely into linear factors when $G$ is an abelian group:
\[\Theta=\prod_{\chi\in\mathrm{Irr}(G)}\sum_{g\in G}\chi(g)x_g.\]
In particular, he illustrated that the characters of $G$ can be recovered from the factorization of $\Theta$ when $G$ is abelian. He also showed that if $G$ is nonabelian, $\Theta$ does not factor into linear factors. In fact, Dedekind believed that the linear factors of $\Theta$ come from the factors of the group determinant of $\norm{G:[G,G]}$ in general, although a proof evaded him. 

The developments that proceeded the initial correspondence between Frobenius and Dedekind are truly remarkable. Within one year, Frobenius succeeded in generalizing finite group characters to all nonabelian groups, as well as discovering several results now considered foundational to the subject of character theory. These developments entirely arose from his work regarding the factorization of the group determinant---a beautiful, yet antiquated approach that seems far removed from the modern perspective. We summarize a few of these results here; however, we refer the reader to \cite{hawkinsorigins} for a concise overview of the work of Dedekind, Frobenius, and the origins of character theory, and to \cite{pioneers} for a thorough invetigation into the origins of representation theory. We briefly mention here that Kenneth Johnson has done considerable work related to the group determinant and representation theory, as well as their connections to S-rings, quasigroups, loops, projective character theory, etc. We refer the reader to \cite{KJgroupdet} for more information.
\begin{thm}[Frobenius]\label{frobenius}
Let $G$ be a finite group with $n$ conjugacy classes $C_1$, $C_2,\dotsc,C_n$. The following statements hold.
\begin{enumerate}[label={\bf(\arabic*)}]
\item The group determinant $\Theta$ has $n$ irreducible factors $\Phi_i$, each monic in $x_1$, and each having degree equal to its multiplicity in the factorization of $\Theta$. That is 
\[\Theta=\prod_{i=1}^n\Phi_i^{d_i},\]
where each $\Phi_i$ is irreducible of degree $d_i$, and monic in $x_1$. 
\item If the condtions $x_{gh}=x_{hg}$ for every $g,h\in G$ are imposed on the indeterminates $x_g$, $g\in G$, then $\Theta$ is a polynomial in the indeterminates $x_1,\dotsc,x_n$, where $x_g=x_i$ for every $g\in C_i$; moreover, $\Theta$ is a product of linear factors:
\[\Theta=\prod_{i=1}^n\left(\,\sum_{j=1}^n \frac{\norm{C_j}\chi_i^j}{d_i}x_j\right)^{\xstrut d_i^2}.\]
\end{enumerate}
\end{thm}
For each $1\le i\le n$, the numbers $\chi_i^j$ give rise to a class function $\chi_i$ on $G$ defined by $\chi_i(g)=\chi_i^j$ if $g\in C_j$, and satisfying $\chi_i(1)=d_i$. These are the characters of $G$ as defined by Frobenius, and they coincide with the current notion of the irreducible characters of $G$. 

The determinant $\Theta$ can be realized from a more modern representation theoretic approach by the formula
\[\Theta=\det\left(\sum_{g\in G}\rho(g)x_g\right),\]
where $\rho$ is the regular representation of $G$ (see \cite{hawkinsorigins}). A similar perspective can be taken to explain the linear factors of the conjugacy class version of the group determinant described in Theorem~\ref{frobenius} (2): If $\mu$ denotes the regular representation of $Z(\mathbb{C}{G})$, then (cf. \cite[Theorem 6.1]{hawkinsorigins})
\[\det\left(\sum_{i=1}^n\mu\bigl(\widehat{C_i}\bigr)x_i\right)=\prod_{i=1}^n\left(\,\sum_{j=1}^n \frac{\norm{C_j}\chi_i^j}{d_i}x_j\right).\]
From this perspective, it is therefore natural to ask if a result similar to Theorem~\ref{frobenius} (2) holds for any other subalgebras of $Z(\mathbb{C}{G})$. The point of this paper is to show that this is, in fact, the case when one considers very specific subalgebras, called supercharacter theories. 

A {\it supercharacter theory} may be defined as a subspace of $Z(\mathbb{C}{G})$ that is a unital subalgebra with respect to two different products: (1) the {\it ordinary product} defined by $\sum_ga_gg\cdot\sum_gb_gg=\sum_{g,h}a_gb_hgh$, and (2) the {\it Hadamard product} defined by  $\sum_ga_gg\circ\sum_gb_gg=\sum_ga_gb_gg$. This is not the definition used by Diaconis--Isaacs  in \cite{ID07}, where the foundations of supercharacter theory were formulated; however it is shown in \cite{AH12} to be equivalent and will prove more useful for our present considerations.

Although we will discuss supercharacter theory in more detail in Section~\ref{supersection}, we must at this point briefly discuss some fundamentals in order to state our main results. By determining a basis of orthogonal idempotents of a supercharacter theory $\mathsf{S}$ of $G$ with respect to the Hadamard product, one obtains a partition $\mathcal{K}$ of $G$ into {\it superclasses}, each of which is a union of conjugacy classes of $G$.  With respect to ordinary product, one obtains a set $\mathrm{BCh}(\mathsf{S})$ of $\norm{\mathcal{K}}$ mutually orthogonal {\it supercharacters} that are constant on the parts of $\mathcal{K}$, the constituents of which encompass all of $\mathrm{Irr}(G)$. When the supercharacter theory $\mathsf{S}$ is given, we will refer to the superclasses as $\mathsf{S}$-{\it classes} and the supercharacters as {\it basic $\mathsf{S}$-characters}.

Our first main result is the following generalization of Theorem~\ref{frobenius} (2).

\begin{introthm}\label{thmA}
Let $\mathsf{S}$ be a supercharacter theory of the group $G$. Let $C\subseteq G$ denote a set of $\mathsf{S}$-class representatives for $G$, and let $x_g$ be independent indeterminates indexed by the elements of $C$. Let $\kappa:G\to C$ be the function that sends an element $g\in G$ to its $\mathsf{S}$-class representative lying in $C$. Let $\Theta_{\mathsf{S}}$ be the determinant $\norm{x^{}_{\kappa(gh^{-1})}}$. Then
\[\Theta_{\mathsf{S}}=\prod_{\chi\in\mathrm{BCh}(\mathsf{S})}\left(\sum_{g\in C}\frac{\chi(g)\norm{\mathrm{cl}_{\mathsf{S}}(g)}}{\chi(1)}\;x_g\right)^{\xstrut\chi(1)}.\]
\end{introthm}

We show that Theorem~\ref{thmA} has somewhat of a converse; this allows us to give a new characterization of supercharacter theories in terms of factorizations of the group determinant. 

\begin{introthm}\label{thmB}
Let $\mathcal{P}$ be a partition of $G$ of length $\ell$ into $G$-invariant parts, and assume that $\{1\}\in\mathcal{P}$. Let $C\subseteq G$ be a set of representatives for the parts of $\mathcal{P}$, and let $x_g$ be independent indeterminates indexed by the elements of $C$.  Let $\kappa:G\to C$ be the function that maps an element of $G$ to its representative in $C$.  Let $\Theta_{\mathcal{P}}$ be the determinant $\norm{x^{}_{\kappa(gh^{-1})}}$.  The partition $\mathcal{P}$ gives the superclasses for a supercharacter theory $\mathsf{S}$ of $G$ if and only if $\Theta_{\mathcal{P}}$ factors into a product of powers of $\ell$ distinct linear factors, each monic in $x_1$. In this event, if 
\[\Theta_{\mathcal{P}}=\prod_{j=1}^\ell\left(\sum_{g\in C}\xi_g^j x_g\right)^{\xstrut m_j}\]
is the factorization of $\Theta_{\mathcal{P}}$ into products of powers of linear factors that are monic in $x_1$, then the basic $\mathsf{S}$-characters are the functions $\chi_j:G\to\mathbb{C}$ defined by
\[\chi_j(g)=\frac{m_j\xi_{\kappa(g)}^j}{\norm{K_g}}.\]
Moreover, $\chi_j(1)=m_j$ for each $1\le j\le \ell$.
\end{introthm}

Finally we show that the linear factors appearing in the factorization of Theorem~\ref{thmA} are coming from the regular representation of the abelian algebra $\mathsf{S}$.

\begin{introthm}\label{regular}
Let $\mathsf{S}\subseteq Z(\mathbb{C}{G})$ be a supercharacter theory of $G$. Write $\mathrm{Cl}(\mathsf{S})=\{C_1,\dotsc,C_n\}$, and let $g_i\in C_i$ for each $i$. Let $x_1,\dotsc,x_n$ be indeterminates. If $\mu$ is the regular representation of $\mathsf{S}$, then
\[\det\left(\sum_{k=1}^n\mu\bigl(\widehat{C_k}\bigr)x_k\right)=\prod_{\chi\in\mathrm{BCh}(\mathsf{S})}\left(\sum_{k=1}^n\frac{\chi(g_k)\norm{C_k}}{\chi(1)}\;x_k\right).\]
\end{introthm}

\section{Supercharacter Theory}\label{supersection}
In this section, we review the basic principles of supercharacter theory and discuss their algebraic structure. The group $G$ will always represent a finite group. For an element $g\in G$, we let $\mathrm{cl}_G(g)$ denote the $G$-conjugacy class of $g$. Given a character $\chi$ of $G$, we let $\mathrm{Irr}(\chi)$ denote the set of irreducible constituents of $\chi$.

Supercharacter theory was born from attempts to find a \enquote{useful} approximation of the complex character theory of the Sylow $p$-subgroups of finite general linear groups over a field of characteristic $p$ (e.g., see \cite{CA95,NY01}). In \cite{ID07}, Diaconis and Isaacs formalized these efforts to describe an approximation of the complex character theory for all algebra groups.
They also proved many fundamental results and laid the groundwork for the subject of supercharacter theory. 

A supercharacter theory, as defined in \cite{ID07}, is a pair $(\mathcal{X},\mathcal{K})$, where $\mathcal{X}$ is a partition of $\mathrm{Irr}(G)$ and $\mathcal{K}$ is partition of $G$ satisfying: (1) $\{1\}\in\mathcal{K}$; (2) $\norm{\mathcal{X}}=\norm{\mathcal{K}}$; (3) for every $X\in\mathcal{X}$, there exists a character $\xi_X$ with $\mathrm{Irr}(\xi_X)\subseteq X$ that is constant on the parts of $\mathcal{K}$. The parts of $\mathcal{K}$ are referred to as superclasses, and the characters $\xi_X$ are called supercharacters. It is shown that $\xi_X$ is the character $\sigma_X=\sum_{\psi\in X}\psi(1)$ up to a scalar,  the principal character appears in its own part of $\mathcal{X}$, and also that each part of $\mathcal{K}$ is a union of conjugacy classes of $G$. We call such a part $G$-invariant. Other notable structural results include: Every irreducible character is a constituent of some supercharacter;  The partitions $\mathcal{X}$ and $\mathcal{K}$ uniquely determine each other. 

Another important observation about supercharacter theories appearing in \cite{ID07} is that the subspace $\mathcal{A}=\mathbb{C}\text{-}\mathrm{span}\{\widehat{K}:K\in\mathcal{K}\}\subseteq Z(\mathbb{C}{G})$, where $\widehat{K}=\sum_{g\in K}g$, is closed under the ordinary product. Since the parts of $\mathcal{K}$ are disjoint, it is also closed under the Hadamard product and so $\mathcal{A}$ is a unital subalgebra with respect to both products. This shows that a supercharacter theory gives rise to a central {\it Schur ring}, a fact first observed by Hendrickson in \cite{AH12}. The reverse is also true; Proposition 2.4 of \cite{AH12} establishes that the map $(\mathcal{X},\mathcal{K})\mapsto\mathbb{C}\text{-}\mathrm{span}\{\widehat{K}:K\in\mathcal{K}\}$ is a bijection from the set of supercharacter theories of a group $G$ to the set of central Schur rings over $G$.

With these considerations, we define a supercharacter theory to be a subspace $\mathsf{S}$ of $Z(\mathbb{C}{G})$ that is a unital subalgebra with respect to both the ordinary and Hadamard products. Such an algebra is semisimple with respect to both products, and hence has a basis of orthogonal idempotents with respect to each product. Since the conjugacy class sums give a basis of orthogonal idempotents of $Z(\mathbb{C}{G})$ with respect to the Hadamard product, one may obtain the superclasses as a basis of orthogonal idempotents of $\mathsf{S}$ with respect to the Hadamard product, as discussed in \cite{AH12}. Similarly, elements $e_\chi$ for $\chi\in\mathrm{Irr}(G)$ defined by the formula $e_\chi=\sum_{g\in G}\chi(1)\chi(g^{-1})g/\norm{G}$, give a basis of orthogonal idempotents of $Z(\mathbb{C}{G})$ with respect to the ordinary product and thus one may obtain a partition $\mathrm{Ch}(\mathsf{S})$ of $\mathrm{Irr}(G)$ such that the set $\{\sum_{\psi\in X}e_\psi:X\in\mathrm{Ch}(\mathsf{S})\}$ is a basis of orthogonal idempotents of $\mathsf{S}$ with respect to the ordinary product. We denote by $\mathrm{Cl}(\mathsf{S})$ the set of superclasses for the supercharacter theory $\mathsf{S}$ and call its elements $\mathsf{S}$-classes. For an element $g\in G$, we let $\mathrm{cl}_{\mathsf{S}}(g)$ denote the $\mathsf{S}$-class containing $g$. We define a set $\mathrm{BCh}(\mathsf{S})$ of supercharacters by $\mathrm{BCh}(\mathsf{S})=\{\sigma_X:X\in\mathrm{Ch}(\mathsf{S})\}$ and call its elements the basic $\mathsf{S}$-characters. 

\section{Group determinants and superclasses}
In this section, we prove Theorem~\ref{thmA} and Theorem~\ref{thmB}. We begin with an alternative statement of Theorem~\ref{frobenius} (2), rewritten using modern language. It is exactly this result that Theorem~\ref{thmA} is meant to generalize.
\begin{thm}[Frobenius]\label{frob}
Let $C\subseteq G$ denote a set of conjugacy class representatives for $G$, and let $x_g$ be independent indeterminates. Let $\kappa:G\to C$ be the function that sends an element $g\in G$ to its conjugacy class representative in $C$. Let $\Theta^\ast$ be the determinant $\norm{x^{}_{\kappa(gh^{-1})}}$. Then
\[\Theta^\ast=\prod_{\psi\in\mathrm{Irr}(G)}\left(\sum_{g\in C}\frac{\psi(g)\norm{\mathrm{cl}_G(g)}}{\psi(1)}\;x_g\right)^{\xstrut\psi(1)^2}.
\]
\end{thm}
Let $a\in Z(\mathbb{C}{G})$. Since $\{e_\psi:\psi\in\mathrm{Irr}(G)\}$ forms a basis of $Z(\mathbb{C}{G})$, where $e_\psi=\sum_{g\in G}\psi(1)\psi(g^{-1})g/\norm{G}$ is the primitive central idempotent associated to $\psi$, one may express
\[a=\sum_{\psi\in\mathrm{Irr}(G)}\omega_\psi(a)e_\psi.\]
It is shown in Chapter 3 of \cite{MI76} that the functions $\omega_\psi:Z(\mathbb{C}G)\to\mathbb{C}$ are algebra homomorphisms, and can be defined on the basis of conjugacy class sums by
$\omega_\psi(\widehat{K})=\norm{K}\psi(g_K)/\psi(1)$, where $K$ is a conjugacy class of $G$ and $g_K\in K$. Using these functions, one may have written the result of Theorem~\ref{frob} as
\[\Theta^\ast=\prod_{\psi\in\mathrm{Irr}(G)}\left(\sum_{g\in C}\omega_\psi\left(\widehat{\mathrm{cl}_G(g)}\right)x_g\right)^{\xstrut\psi(1)^2}.\] 

In order to obtain a supercharacter theoretic version of Theorem~\ref{frob}, we need the following result, which explicitly gives the values that the $\omega$ functions take on superclass sums.
\begin{lem}\label{omega}
Let $\mathsf{S}$ be a supercharacter theory of $G$, let $K$ be an $\mathsf{S}$-class, and let $g\in K$. For each $\chi\in\mathrm{BCh}(\mathsf{S})$ and for every $\psi\in\mathrm{Irr}(\chi)$, we have
\[\omega_\psi(\widehat{K})=\frac{\norm{K}\chi(g)}{\chi(1)}.\]
\end{lem}
\begin{proof}
Since $\widehat{K}\in Z(\mathbb{C}G)$, we have
\[\widehat{K}=\sum_{\psi\in\mathrm{Irr}(G)}\omega_\psi(\widehat{K})e_\psi,\]
where $e_\psi=\sum_{g\in G}\psi(1)\psi(g^{-1})g/\norm{G}$ is the primitive central idempotent associated to $\psi$. The space $\mathbb{C}$-$\mathrm{span}\{\widehat{L}:L\in\mathrm{Cl}(\mathsf{S})\}$ has the set $\{E_\chi:\chi\in\mathrm{BCh}(\mathsf{S})\}$ as a basis, where $E_\chi=\sum_{\psi\in\mathrm{Irr}(\chi)}e_\psi$; it follows that $\omega_\psi(\widehat{K})$ is constant as $\psi$ ranges over $\mathrm{Irr}(\chi)$ for each $\chi\in\mathrm{BCh}(\mathsf{S})$. In the proof of Theorem 2.4 of \cite{ID07}, it is shown that this number is exactly $\norm{K}\chi(g)/\chi(1)$. In fact, this is exactly the coefficient of $E_\chi$ when expressing $\widehat{K}$ in terms of the $\{E_\chi:\chi\in\mathrm{BCh}(\mathsf{S})\}$ basis.
\end{proof}

Given the basic $\mathsf{S}$-characters of a supercharacter theory $\mathsf{S}$ of $G$, one may recover the $\mathsf{S}$-classes rather easily: two elements $g$ and $h$ of $G$ lie in the same $\mathsf{S}$-class if and only if $\chi(g)=\chi(h)$ for every $\chi\in\mathrm{BCh}(\mathsf{S})$. In light of Lemma~\ref{omega}, we obtain a similar procedure for determining $\mathrm{BCh}(\mathsf{S})$ from $\mathrm{Cl}(\mathsf{S})$.
\begin{lem}\label{omegasct}
Let $\mathcal{P}$ be a partition of $G$ into $G$-invariant parts, and assume that $\{1\}\in\mathcal{P}$. For each $g\in G$, let $K_g$ denote the part of $\mathcal{P}$ containing $g$. Define an equivalence relation $\sim$ on $\mathrm{Irr}(G)$ by defining $\chi\sim\psi$ if and only if $\omega_\chi(\widehat{K_g})=\omega_\psi(\widehat{K_g})$ for all $g\in C$. The partition $\mathcal{P}$ gives the superclasses for a supercharacter theory of $G$ if and only if $\mathrm{Irr}(G)$ has exactly $\norm{\mathcal{P}}$ equivalence classes under $\sim$. In this event, if $\mathcal{X}$ denotes the set of equivalence classes of $\mathrm{Irr}(G)$ under $\sim$, then we have $\mathrm{BCh}(\mathsf{S})=\{\sigma_X:X \in\mathcal{X}\}$.
\end{lem}
\begin{proof}
The one direction is covered in Lemma~\ref{omega}. So assume that $\mathcal{P}$ is a partition of $G$ into $G$-invariant parts satisfying $\{1\}\in\mathcal{P}$. Write $\mathsf{S}=\mathbb{C}\text{-}\mathrm{span}\{\widehat{K}:K\in\mathcal{P}\}$. Let $\mathcal{X}$ denote the set of equivalence classes of $\mathrm{Irr}(G)$ under the equivalence relation $\sim$, and assume that $\norm{\mathcal{X}}=\norm{\mathcal{P}}$. It follows that for every $K\in\mathcal{P}$, there exist coefficients $\Omega_X(\widehat{K})$ so that $\widehat{K}=\sum_{X\in\mathcal{X}}\Omega_X(\widehat{K})F_X$, where $F_X=\sum_{\psi\in X}e_\psi$. Therefore, $\{F_X:X\in\mathcal{X}\}$ is a basis of idempotents for $\mathsf{S}$ with respect to the ordinary product. The result follows by \cite[Proposition 2.4]{AH12}.
\end{proof}
We now prove Theorem~\ref{thmA}, which we restate here.
\setcounter{intro}{0}
\begin{introthm}\label{sfrob1}
Let $\mathsf{S}$ be a supercharacter theory of the group $G$. Let $C\subseteq G$ denote a set of $\mathsf{S}$-class representatives for $G$, and let $x_g$ be independent indeterminates indexed by the elements of $C$. Let $\kappa:G\to C$ be the function which sends an element $g\in G$ to its $\mathsf{S}$-class representative lying in $C$. Let $\Theta_{\mathsf{S}}$ be the determinant $\norm{x^{}_{\kappa(gh^{-1})}}$. Then
\[\Theta_{\mathsf{S}}=\prod_{\chi\in\mathrm{BCh}(\mathsf{S})}\left(\sum_{g\in C}\frac{\chi(g)\norm{\mathrm{cl}_{\mathsf{S}}(g)}}{\chi(1)}\;x_g\right)^{\xstrut\chi(1)}.\]
\end{introthm}
\begin{proof}
Since $\omega_\psi$ is an algebra homomorphism $Z(\mathbb{C}(G))\to\mathbb{C}$, it follows from Theorem~\ref{frob} that
\begin{align*}
\Theta_{\mathsf{S}}=\prod_{\psi\in\mathrm{Irr}(G)}\left(\sum_{g\in C}\omega_\psi\left(\widehat{\mathrm{cl}_{\mathsf{S}}(g)}\right)x_g\right)^{\xstrut\psi(1)^2}.
\end{align*}
Also, we have that $\omega_\psi(\widehat{K})=\omega_\chi(\widehat{K})$ for each $K\in\mathrm{Cl}(\mathsf{S})$ whenever $\psi\in\mathrm{Irr}(\chi)$ by Lemma~\ref{omega}. 
So
\begin{align*}\Theta_{\mathsf{S}}&=\prod_{\psi\in\mathrm{Irr}(G)}\left(\sum_{g\in C}\omega_\psi\left(\widehat{\mathrm{cl}_{\mathsf{S}}(g)}\right)x_g\right)^{\xstrut\psi(1)^2}\\
&=\prod_{\chi\in\mathrm{BCh}(\mathsf{S})}\prod_{\psi\in\mathrm{Irr}(\chi)}\left(\sum_{g\in C}\omega_\psi\left(\widehat{\mathrm{cl}_{\mathsf{S}}(g)}\right)x_g\right)^{\xstrut\psi(1)^2}\\
&=\prod_{\chi\in\mathrm{BCh}(\mathsf{S})}\prod_{\psi\in\mathrm{Irr}(\chi)}\left(\sum_{g\in C}\omega_\psi\left(\widehat{\mathrm{cl}_{\mathsf{S}}(g)}\right)x_g\right)^{\xstrut\psi(1)^2}\\
&=\prod_{\chi\in\mathrm{BCh}(\mathsf{S})}\left(\sum_{g\in C}\omega_\psi\left(\widehat{\mathrm{cl}_{\mathsf{S}}(g)}\right)x_g\right)^{\:\:\:\:\sum\limits_{\mathclap{\psi\in\mathrm{Irr}(\chi)}}\psi(1)^2}\\
&=\prod_{\chi\in\mathrm{BCh}(\mathsf{S})}\left(\sum_{g\in C}\frac{\chi(g)\norm{\mathrm{cl}_{\mathsf{S}}(g)}}{\chi(1)}\;x_g\right)^{\xstrut\chi(1)},
\end{align*}
as desired.
\end{proof}

Observe that the set $\{\psi(1)\psi:\psi\in\mathrm{Irr}(G)\}$ gives the supercharacters for the supercharacter theory $Z(\mathbb{C}{G})$ of $G$. Specializing Theorem~\ref{thmA} to this supercharacter theory yields the second statement of Theorem~\ref{frobenius}. In particular, Theorem~\ref{thmA} really is a generalization of the second statement of Theorem~\ref{frobenius}.

We now present some examples that illustrate  the conclusion of Theorem~\ref{thmA}. The supercharacter theory $Z(\mathbb{C}{G})$ of $G$ is the supercharacter theory that is contained in every other supercharacter theory of $G$. There is also a supercharacter theory that contains all others. Specifically this is the supercharacter theory with superclasses $\{\{1\},G\setminus\{1\}\}$. The next example is exactly Theorem~\ref{thmA} specialized to this supercharacter theory.
\begin{exam}
Let $G$ be any finite group, and let $\mathsf{S}$ be the supercharacter theory of $G$ with $\mathrm{Cl}(\mathsf{S})=\{\{1\},G\setminus\{1\}\}$. Then $\mathrm{BCh}(\mathsf{S})=\{\mathbbm{1},\rho-\mathbbm{1}\}$, where $\rho$ denotes the regular character of $G$. Let $g$ be a nonidentity element of $G$. We will first compute $\Theta_{\mathsf{S}}$ without appealing to Theorem~\ref{sfrob1}. The group matrix corresponding to $\mathsf{S}$ is the matrix
\[X=\left(\begin{array}{cccc}x_1&x_g&\cdots&x_g\\
x_g&x_1&&\multirow{2}{*}{\vdots}\\
\vdots&&\ddots&\\
x_g&\multicolumn{2}{c}{\cdots}&x_1
\end{array}\right).\]
Let $Y=X+(x_g-x_1)\cdot I$, the matrix with all entries $x_g$. It is not difficult to see that the eigenvalues of $Y$ are 0 with multiplicity $\norm{G}-1$, and $\norm{G}x_g$ with multiplicity 1. Therefore the eigenvalues of $X$ are $x_1-x_g$ with multiplicity $\norm{G}-1$, and $\norm{G}x_g+(x_1-x_g)=x_1+(\norm{G}-1)x_g$ with multiplicity 1. Therefore
\[\Theta_{\mathsf{S}}=\det(X)=(x_1+(\norm{G}-1)x_g)(x_1-x_g)^{\norm{G}-1}.\]
We see that this agrees with Theorem~\ref{sfrob1}. Indeed, the linear factor contributed by the principal character is $x_1+\norm{\mathrm{cl}_{\mathsf{S}}(g)}x_g=x_1+(\norm{G}-1)x_g$, and the term contributed by $\chi=\rho-\mathbbm{1}$ is $\bigl(x_1+\frac{\chi(g)\norm{\mathrm{cl}_{\mathsf{S}}(g)}}{\chi(1)}x_g\bigr)^{\chi(1)}=(x_1-x_g)^{\norm{G}-1}$.
\end{exam}

We now give another example that illustrates the conclusion of Theorem~\ref{thmA}. In this example we consider a supercharacter theory of a cyclic group coming from the inversion automorphism. Perhaps interesting in its own right, we consequently obtain a formula for the determinant of specific type of symmetric Toeplitz matrix.
\begin{exam}
Let $G=\inner{a}$ be a cyclic group of order $n$. Let $\mathsf{S}$ be the supercharacter theory of $G$ coming from the action of the inversion automorphism. 

First let $n$ be odd. Then we have $\mathrm{BCh}(\mathsf{S})=\{\mathbbm{1}\}\cup\{\psi+\overline{\psi}:\psi\ne\mathbbm{1}\}$, and $\mathrm{Cl}(\mathsf{S})=\{\{1\}\}\cup\{\{a^k,a^{-k}\}:1\le k\le (n-1)/2\}$. It follows that $\{1,a,a^2,\dotsc,a^{(n-1)/2}\}$ is a set of representatives of the $\mathsf{S}$-classes. The irreducible characters of $G$ are the characters $\psi_j$ defined by $\psi_j(a)=e^{2\pi i j/n}$ for $0\le j\le n-1$. The value that $\psi_j+\overline{\psi_j}$ takes on the $\mathsf{S}$-class $\{a^k,a^{-k}\}$ is $2\cos(2\pi jk/n)$. The group matrix corresponding to $\mathsf{S}$ has $x_1$ on the diagonal, $x_{a^{k+1}}$ along the $k$-diagonal for $1\le k\le (n-1)/2$, and is symmetric. We have
\[\Theta_{\mathsf{S}}=\Delta\cdot\prod_{j=1}^{(n-1)/2}\left(\,x_1+\sum_{k=1}^{(n-1)/2}2\cos\left(\frac{2\pi k}{n}\right)x_{a^k}\right)^{\xstrut 2},\]
where
\[\Delta=x_1+2x_a+\dotsb+2x_{a^{(n-1)/2}}.\]

Now let $n$ be even. Let $\lambda\in\mathrm{Irr}(G)$ be the character defined by $\lambda(a)=-1$. Then $\mathrm{BCh}(\mathsf{S})=\{\mathbbm{1},\lambda\}\cup\{\psi+\overline{\psi}:\psi\ne\mathbbm{1},\lambda\}$, and $\mathrm{Cl}(\mathsf{S})=\{\{1\},\{a^{n/2}\}\}\cup\{\{a^k,a^{-k}\}:1\le k< n/2\}$. It follows that $\{1,a,a^2,\dotsc,a^{n/2}\}$ is a set of representatives of the $\mathsf{S}$-classes. As before, the irreducible characters of $G$ are the characters $\psi_j$ defined above. Observe that $\lambda=\psi_{n/2}$. Therefore, the value that $\psi_j+\overline{\psi_j}$ takes on the $\mathsf{S}$-class $\{a^k,a^{-k}\}$ is $2\cos(2\pi jk/n)$ if $j\ne n/2$, and is $(-1)^k$ if $j=n/2$. The group matrix corresponding to $\mathsf{S}$ has $x_1$ on the diagonal, $x_{a^{k+1}}$ along the $k$-diagonal for $1\le k\le n/2-1$, and is symmetric. We have
\[\Theta_{\mathsf{S}}=\Lambda\cdot \prod_{j=1}^{n/2-1}\left(\,x_1+(-1)^jx_{a^{n/2}}+\sum_{k=1}^{n/2-1}2\cos\left(\frac{2\pi kj}{n}\right)x_{a^k}\right)^{\xstrut 2},\]
where 
\[\Lambda=\biggl(x_1+x_{a^{n/2}}+2\sum_{k=1}^{n/2-1}x_{a^k}\biggr)\biggl(x_1+x_{a^{n/2}}+2\sum_{k=1}^{n/2-1}(-1)^kx_{a^k}\biggr).\]
\end{exam}

Theorem~\ref{thmA} is one direction of Theorem~\ref{thmB}. The next theorem asserts the other direction, and serves as a somewhat of a converse to Theorem~\ref{sfrob1}.
\begin{thm}\label{sfrob2}
Let $\mathcal{P}$ be a partition of $G$ of length $\ell$ into $G$-invariant parts, and assume that $\{1\}\in\mathcal{P}$. Let $C\subseteq G$ be a set of representatives for the parts of $\mathcal{P}$, and let $x_g$ be independent indeterminates indexed by the elements of $C$.  For each $g\in G$, let $K_g$ denote the part of $\mathcal{P}$ containing $g$. Let $\kappa:G\to C$ be the function that maps an element of $G$ to its representative in $C$.   Let $\Theta_{\mathcal{P}}$ be the determinant $\norm{x^{}_{\kappa(gh^{-1})}}$. Suppose that $\Theta_{\mathcal{P}}$ factors into a product of powers of $\ell$ distinct linear factors, each monic in $x_1$, say 
\[\Theta_{\mathcal{P}}=\prod_{j=1}^\ell\left(\sum_{g\in C}\xi_g^j x_g\right)^{\xstrut m_j}.\] Then $\mathcal{P}$ gives the $\mathsf{S}$-classes for a supercharacter theory $\mathsf{S}$ of $G$ with basic $\mathsf{S}$-characters given by the functions $\chi_j:G\to\mathbb{C}$ defined by
\[\chi_j(g)=\frac{m_j\xi_{\kappa(g)}^j}{\norm{K_g}}.\]
\end{thm}
\begin{proof}
Choose a set of conjugacy class representatives $\tilde{C}\subseteq G$ satisfying $C\subseteq\tilde{C}$, and let $\tilde{\kappa}:G\to \tilde{C}$ be the function which sends an element $g\in G$ to its conjugacy class representative lying in $\tilde{C}$. As in the statement of Theorem~\ref{frob}, let $\Theta^\ast$ be the determinant $\norm{x^{}_{\tilde{\kappa}(gh^{-1})}}$; then we have
\[\Theta^\ast=\prod_{\psi\in\mathrm{Irr}(G)}\left(\,\sum_{g\in \tilde{C}}\omega_\psi\left(\widehat{\mathrm{cl}_G(g)}\right)x_g\right)^{\xstrut\psi(1)^2},\]
and each of these linear factors is monic in $x_1$. Let $\mathscr{F}:\mathbb{C}[x_g:g\in \tilde{C}]\to\mathbb{C}[x_g:x\in C]$ be the $\mathbb{C}$-algebra homomorphism defined $\mathbb{C}$-linearly by $x_g\mapsto x_{\kappa(g)}$, so that $\Theta_{\mathsf{S}}=\mathscr{F}(\Theta^\ast)$. For each $1\le j\le\ell$, write $\lambda_j=\sum_{g\in C}\xi_g^j x_g$. Since $\mathbb{C}[x_g:x\in C]$ is a unique factorization domain, and each $\lambda_j$ is monic in $x_1$, we must have that $\lambda_j=\mathscr{F}(\mu)$ for some linear factor of $\Theta^\ast$ for each $1\le j\le \ell$. For each $1\le j\le\ell$, define 
\[X_j=\left\{\psi\in\mathrm{Irr}(G):\mathscr{F}\left(\,\sum_{g\in \tilde{C}}\omega_\psi\left(\widehat{\mathrm{cl}_G(g)}\right)x_g\right)=\lambda_j\right\}.\]
Now fix $1\le j\le \ell$. Since
\begin{align*}
\MoveEqLeft[4]\mathscr{F}\left(\,\sum_{g\in \tilde{C}}\omega_\psi\left(\widehat{\mathrm{cl}_G(g)}\right)x_g\right)=\sum_{g\in C}\left(\,\sum_{\substack{h\in \tilde{C}\\\kappa(h)=g}}\omega_\psi\left(\widehat{\mathrm{cl}_G(h)}\right)\right)x_g\\
&=\sum_{g\in C}\omega_\psi\left(\,\sum_{\substack{h\in \tilde{C}\\\kappa(h)=g}}\widehat{\mathrm{cl}_G(h)}\right)x_g=\sum_{g\in C}\omega_\psi\left(\widehat{K_g}\right)x_g,
\end{align*}
for each $\psi\in\mathrm{Irr}(G)$, we have 
\[X_j=\left\{\psi\in\mathrm{Irr}(G):\omega_\psi\left(\widehat{K_g}\right)=\xi_g^j\ \,\text{for every $g\in C$}\right\}.\]
By Lemma~\ref{omegasct},  the partition $\mathcal{P}$ gives the superclasses for a supercharacter theory of $G$, and the supercharacters for that supercharacter theory are exactly the characters $\sigma^{}_{X_i}$, $1\le i\le \ell$. Also note that we must have
\begin{align*}
\lambda_j^{m_j}&=\mathscr{F}\left(\prod_{\psi\in X_i}\left(\sum_{g\in \tilde{C}}\omega_\psi\left(\widehat{\mathrm{cl}_G(g)}\right)x_g\right)^{\xstrut\psi(1)^2}\right)\\
&=\prod_{\psi\in X_j}\mathscr{F}\left(\sum_{g\in \tilde{C}}\omega_\psi\left(\widehat{\mathrm{cl}_G(g)}\right)x_g\right)^{\xstrut\psi(1)^2}\\
&=\prod_{\psi\in X_j}\lambda_j^{\psi(1)^2}=\lambda_j^{\sum_{\psi\in X_j}\psi(1)^2}=\lambda_j^{\sigma^{}_{X_j}(1)},
\end{align*}
so it follows that $m_j=\sigma_{X_j}(1)$. Hence Lemma~\ref{omega} implies that
\[\frac{m_j\xi_{\kappa(g)}^j}{\norm{K_g}}=\frac{\sigma^{}_{X_j}(1)\,\omega_\psi\left(\widehat{K_g}\right)}{\norm{K_g}}=\sigma^{}_{X_j}(g),\]
and this completes the proof.
\end{proof}
Combining Theorems~\ref{sfrob1} and \ref{sfrob2} yields Theorem~\ref{thmB}, which we restate here.

\begin{introthm}
Let $\mathcal{P}$ be a partition of $G$ of length $\ell$ into $G$-invariant parts, and assume that $\{1\}\in\mathcal{P}$. Let $C\subseteq G$ be a set of representatives for the parts of $\mathcal{P}$, and let $x_g$ be independent indeterminates indexed by the elements of $C$.  Let $\kappa:G\to C$ be the function that maps an element of $G$ to its representative in $C$.  Let $\Theta_{\mathcal{P}}$ be the determinant $\norm{x^{}_{\kappa(gh^{-1})}}$.  The partition $\mathcal{P}$ gives the superclasses for a supercharacter theory $\mathsf{S}$ of $G$ if and only if $\Theta_{\mathcal{P}}$ factors into a product of powers of $\ell$ distinct linear factors, each monic in $x_1$. In this event, if 
\[\Theta_{\mathcal{P}}=\prod_{j=1}^\ell\left(\sum_{g\in C}\xi_g^j x_g\right)^{\xstrut m_j}\]
is the factorization of $\Theta_{\mathcal{P}}$ into products of powers of linear factors that are monic in $x_1$, then the basic $\mathsf{S}$-characters are the functions $\chi_j:G\to\mathbb{C}$ defined by
\[\chi_j(g)=\frac{m_j\xi_{\kappa(g)}^j}{\norm{K_g}}.\]
Moreover, $\chi_j(1)=m_j$ for each $1\le j\le \ell$.
\end{introthm}

\section{Relationship to the regular representation}
We now work toward proving Theorem~\ref{regular}. In order to accomplish this, we will need to be able to decompose products of superclass sums. A well known formula can be used in the case that the superclasses are the usual conjugacy classes (see Exercise 3.9 of \cite{MI76}) by a straightforward application of column orthogonality. We will employ the same method here --- the only difference is that we will need a version of column orthogonality for supercharacters.
\begin{thm}[{\normalfont \cite[Theorem 3.3]{SB18nil}}]\label{column}Let $\mathsf{S}$ be a supercharacter theory of $G$, and let $g,h\in G$. Then
\[ \sum_{\chi\in\mathrm{Ch}(\mathsf{S})}\frac{\chi(g)\overline{\chi(h)}}{\chi(1)}=\frac{\norm{G}}{\norm{\mathrm{cl}_{\mathsf{S}}(g)}}\]
if $h\in\mathrm{cl}_{\mathsf{S}}(g)$, and is 0 otherwise.
\end{thm}

\begin{cor}\label{constants}
Let $\mathsf{S}$ be a supercharacter theory of $G$. For each $K\in\mathrm{Cl}(\mathsf{S})$, let $g_K\in\mathrm{Cl}(\mathsf{S})$. Then for each $K,L\in\mathrm{Cl}(\mathsf{S})$
\[\widehat{K}\widehat{L}=\sum_{C\in\mathrm{Cl}(\mathsf{S})}\left(\frac{\norm{K}\norm{L}}{\norm{G}}\sum_{\chi\in\mathrm{BCh}(\mathsf{S})}\frac{\chi(g_K)\chi(g_L)\overline{\chi(g_C)}}{\chi(1)^2}\right)\widehat{C}.\]
\end{cor}
\begin{proof}
For each $\chi\in\mathrm{BCh}(\mathsf{S})$ and each $K\in\mathrm{Cl}(\mathsf{S})$, let $\Omega_\chi(\widehat{K})$ denote the  coefficient of $E_\chi$ when expressing $\widehat{K}$ in terms of the $\{E_\chi:\chi\in\mathrm{BCh}(\mathsf{S})\}$ basis. Recall that each $\Omega_\chi$ is an algebra homomorphism and that $\Omega_\chi(\widehat{K})=\chi(g_K)\norm{K}/\chi(1)$. Let $K,L\in\mathrm{Cl}(\mathsf{S})$ and write $\widehat{K}\widehat{L}=\sum_{C\in\mathrm{Cl}(\mathsf{S})}a_{K,L,C}\widehat{C}$. Then
\begin{align*}
\norm{K}\norm{L}\sum_{\chi\in\mathrm{BCh}(\mathsf{S})}\frac{\chi(g_K)\chi(g_L)}{\chi(1)^2}&=\sum_{\chi\in\mathrm{BCh}(\mathsf{S})}\Omega_\chi(\widehat{K})\Omega_\chi(\widehat{L})\\
&=\sum_{\chi\in\mathrm{BCh}(\mathsf{S})}\Omega_\chi(\widehat{K}\widehat{L})\\
&=\sum_{M\in\mathrm{Cl}(\mathsf{S})}a_{K,L,M}\sum_{\chi\in\mathrm{BCh}(\mathsf{S})}\Omega_\chi(\widehat{M})\\
&=\sum_{M\in\mathrm{Cl}(\mathsf{S})}a_{K,L,M}\norm{M}\sum_{\chi\in\mathrm{BCh}(\mathsf{S})}\frac{\chi(g_M)}{\chi(1)}.
\end{align*}
Fix $C\in\mathrm{Cl}(\mathsf{S})$. Then
\begin{align*}
\MoveEqLeft[5]\norm{K}\norm{L}\sum_{\chi\in\mathrm{BCh}(\mathsf{S})}\frac{\chi(g_K)\chi(g_L)\overline{\chi(g_C)}}{\chi(1)^2}\\&=\sum_{M\in\mathrm{Cl}(\mathsf{S})}a_{K,L,M}\norm{M}\sum_{\chi\in\mathrm{BCh}(\mathsf{S})}\frac{\chi(g_M)\overline{\chi(g_C)}}{\chi(1)}\\
&=\sum_{M\in\mathrm{Cl}(\mathsf{S})}a_{K,L,M}\norm{M}\delta_{M,C}\frac{\norm{G}}{\norm{C}},
\end{align*}
by Theorem~\ref{column}. The right hand side is just $\norm{G}a_{K,L,C}$, whence the result.
\end{proof}

We will also need the following result, due to Frobenius. The statement we give here is taken from Hawkins' expository paper \cite{hawkinsorigins}.
\begin{thm}[{\normalfont Frobenius; cf. \cite[Theorem 5.3]{hawkinsorigins}}]\label{regfrob}
Let the $n^2m$ numbers $a_{ijk}$ {\upshape($i,j=1,2,\dotsc,n;\ k=1,2,\dotsb,m$\/)} satisfy
\[\sum_{s=1}^na_{psj}a_{sqk}=\sum_{s=1}^na_{psk}a_{spj}.\]
Then if $A_k=(a_{ijk})$, $k=1,2,\dotsb,m$, there exists numbers $r_k^{(s)}$ {\upshape($s=1,2,\dotsc,n;\ k=1,2,\dotsb,m$\/)} such that
\[\det\left(\,\sum_{k=1}^mA_kx_k\right)=\prod_{s=1}^n\left(\,\sum_{k=1}^m r_k^{(s)}x_k\right).\]
Furthermore, when $m=n$ and $a_{ijk}=a_{ikj}$, the $r^{(s)}_k$ are the only solutions to 
\[r_j^{(s)}r_k^{(s)}=\sum_{i=1}^na_{ijk}r_i^{(s)}.\]
\end{thm}

We may now prove Theorem~\ref{regular}, which we restate here.

\begin{introthm}
Let $\mathsf{S}\subseteq Z(\mathbb{C}{G})$ be a supercharacter theory of $G$. Write $\mathrm{Cl}(\mathsf{S})=\{C_1,\dotsc,C_n\}$, and let $g_i\in C_i$ for each $i$. Let $x_1,\dotsc,x_n$ be indeterminates. If $\mu$ is the regular representation of $\mathsf{S}$, then
\[\det\left(\sum_{k=1}^n\mu\bigl(\widehat{C_k}\bigr)x_k\right)=\prod_{\chi\in\mathrm{BCh}(\mathsf{S})}\left(\sum_{k=1}^n\frac{\chi(g_k)\norm{C_k}}{\chi(1)}\;x_k\right).\]
\end{introthm}

\begin{proof}
For each $1\le j,k\le n$, define the integers $a_{ijk}$ by
\[\widehat{C_j}\widehat{C_k}=\sum_{i=1}^na_{ijk}\widehat{C_i}.\]
Write $A_k= (a_{ijk})$. Observe that $\mu(\widehat{C_k})=A_k$ for each $k$; i.e. the function $\widehat{C_k}\mapsto A_k$ is the right regular representation of $\mathsf{S}$.
Since $\mathsf{S}$ is abelian, $A_kA_j=A_jA_k$ and $\widehat{C_j}\widehat{C_k}=\widehat{C_k}\widehat{C_j}$ for every $1\le j,k\le n$. In particular, the numbers $a_{ijk}$ satisfy all of the conditions of Theorem~\ref{regfrob}. 

Write $\mathrm{BCh}(\mathsf{S})=\{\chi_i:1\le i\le n\}$. For each $1\le k\le n$, define the numbers $r_k^{(s)}$ by
\[r_k^{(s)}=\frac{\chi_s(g_k)\norm{C_k}}{\chi_s(1)}.\]
Observe that $\inner{\chi,\chi}=\chi(1)$ if $\chi\in\mathrm{BCh}(\mathsf{S})$. Also note that 
\[a_{ijk}=\frac{\norm{C_j}\norm{C_k}}{\norm{G}}\sum_{t=1}^n\frac{\chi_t(g_j)\chi_t(g_k)\overline{\chi_t(g_i)}}{\chi_t(1)^2},\] for each $1\le i,j,k\le n$ by Corollary~\ref{constants}. Thus
\begin{align*}
\sum_{s=1}^na_{ijk}r_i^{(s)}&=\left(\,\sum_{i=1}^n\frac{\norm{C_j}\norm{C_k}}{\norm{G}}\sum_{t=1}^n\frac{\chi_t(g_j)\chi_t(g_k)\overline{\chi_t(g_i)}}{\chi_t(1)^2}\right)\cdot \frac{\chi_s(g_i)\norm{C_i}}{\chi_s(1)}\\
&=\left(\,\sum_{t=1}^n\frac{\norm{C_j}\norm{C_k}\chi_t(g_j)\chi_t(g_k)}{\chi_t(1)^2\chi_s(1)}\right)\cdot \frac{1}{\norm{G}}\sum_{i=1}^n\norm{C_i}\chi_s(g_i)\overline{\chi_t(g_i)}\\
&=\left(\,\sum_{t=1}^n\frac{\norm{C_j}\norm{C_k}\chi_t(g_j)\chi_t(g_k)}{\chi_t(1)^2\chi_s(1)}\right)\cdot \delta_{s,t}\chi_t(1)\\
&=\frac{\chi_s(g_j)\chi_s(g_k)\norm{C_j}\norm{C_k}}{\chi_s(1)^2}\\
&=r_j^{(s)}r_k^{(s)}
\end{align*}
for every $1\le j,k\le n$. Hence the equation
\[r_j^{(s)}r_k^{(s)}=\sum_{i=1}^na_{ijk}r_i^{(s)}\]
holds for every $j,k$ and so it follows from Theorem~\ref{regfrob} that 
\begin{align*}
\det\left(\sum_{k=1}^n\mu\bigl(\widehat{C_k}\bigr)x_k\right)&=\det\left(\,\sum_{k=1}^mA_kx_k\right)=\prod_{s=1}^n\left(\,\sum_{k=1}^m r_k^{(s)}x_k\right)\\
&=\prod_{\chi\in\mathrm{BCh}(\mathsf{S})}\left(\sum_{k=1}^n\frac{\chi(g_k)\norm{C_k}}{\chi(1)}\;x_k\right),
\end{align*}
as desired.
\end{proof}

\end{document}